\newif\ifextras
\ifextras
\fi 

\documentclass{amsart}

\usepackage{amsmath}
\usepackage{amsthm}
\usepackage{amssymb}
\usepackage{enumerate}
\usepackage{xcolor}
\usepackage{graphicx}
\usepackage{caption}
\usepackage{subcaption}
\usepackage{wrapfig}
\input{insbox}
\usepackage{hyperref}
\hypersetup{
    colorlinks=true,
    linkcolor=black,
    filecolor=magenta,      
    urlcolor=blue,
}
\newtheorem{definition}{Definition}[section]
\newtheorem{theorem}[definition]{Theorem}
\newtheorem{lemma}[definition]{Lemma}
\newtheorem{corollary}[definition]{Corollary}

\theoremstyle{definition}
\newtheorem{remark}[definition]{Remark}
\newtheorem{construction}[definition]{Construction}

\newcommand{\Leb}{\mathrm{Leb}}

\newcommand{\Mp}{\mathcal{M}_{1}}
\newcommand{\Mfin}{\mathcal{M}_{fin}}
\newcommand{\B}{\mathcal{B}}

\newcommand{\CI}{\mathrm{C}(\mathrm{I})}
\newcommand{\CR}{\mathrm{C}(\mathbb{R})} 
\newcommand{\RDSI}{\mathrm{RDS}(\mathrm{I})}
\newcommand{\RDSR}{\mathrm{RDS}(\mathbb{R})} 
\newcommand{\XI}{\mathrm{X}(\mathrm{I})}

\newcommand{\NMalpha}{\mathcal{N}_{M, \alpha}}

\newcommand{\id}{\mathop{\mathrm{id}}}
\newcommand{\supp}{\mathop{\mathrm{supp}}}
\newcommand{\Orb}{\mathop{\mathrm{Orb}}}

\begin{document}

\title{Typical behaviour of random interval homeomorphisms}

\author{Jaroslav Brad\'ik}
\address{Mathematical Institute, Silesian University in Opava, Na Rybničku 1, 74601, Opava, Czech Republic}
\email{jaroslav.bradik@math.slu.cz}

\author{Samuel Roth}
\address{Mathematical Institute, Silesian University in Opava, Na Rybničku 1, 74601, Opava, Czech Republic}
\email{samuel.roth@math.slu.cz}



\makeatletter
\@namedef{subjclassname@2020}{%
  \textup{2020} Mathematics Subject Classification}
\makeatother

\thanks{Research was funded by institutional support for the development of research organizations (I\v{C}47813059) and by grant SGS 18/2019}

\begin{abstract}
We consider the typical behaviour of random dynamical systems of order-preserving interval homeomorphisms with a positive Lyapunov exponent condition at the endpoints. Our study removes any requirement for continuous differentiability save the existence of finite derivatives of the homeomorphisms at the endpoints of the interval. We construct a suitable Baire space structure for this class of systems. Generically within this Baire space, we show that the stationary measure is singular with respect to the Lebesgue measure, but has full support on $[0,1]$. This provides an answer to a question raised by Alsed\`a and Misiurewicz.
\end{abstract}

\maketitle


\section{Introduction}

In \cite{AlsedaMisiurewicz}, the authors raise the question of properties of stationary measures for typical random dynamical systems of interval homeomorphisms. First, let's examine the possible meaning of a \emph{typical} random dynamical system. One viewpoint could be that the typical systems are the ones which form a residual subset of a Baire space of all the random dynamical systems under consideration. Another viewpoint is to enlarge the class of random dynamical systems under consideration to include as many systems as possible by imposing as few constraints on it as possible. We adopt both viewpoints and construct a Baire space of random dynamical systems with a limited number of constraints and look for a residual subset of this space.

Let us summarize the question raised in \cite{AlsedaMisiurewicz} and the attempts to answer it so far. The random iterated function systems in \cite{AlsedaMisiurewicz} use piecewise linear homeomorphisms of $[0,1]$ onto itself with common fixed points at $0$ and $1$. Clearly these systems carry trivial stationary measures concentrated at the fixed endpoints. However, under a positive Lyapunov exponent condition (which makes both endpoints repelling on average), the authors prove the existence and uniqueness of a stationary measure $\mu$ on the open interval $(0,1)$. The conjecture made in \cite{AlsedaMisiurewicz} is that for the “typical” system, the stationary measure $\mu$ is singular. 

A partial solution is given in \cite{BaranskiSpiewak}, where the authors continue to work with piecewise linear homeomorphisms and prove singularity of $\mu$ in the so-called resonant case. Moreover, they give a specific example of a resonant system for which $\mu$ is not only singular, but has full support in $[0,1]$. 

In \cite{GharaeiHomburg}, the existence and uniqueness of a stationary measure $\mu$ on $(0,1)$ is proved for a different class of iterated function systems of the form $f = (f_0, f_1, p)$, where $f_0$, $f_1$ are twice continuously differentiable homeomorphisms on the unit interval such that $f_0(x) < x, f_1(x) > x$ on $(0, 1)$, and $p \in (0, 1)$ determines the probabilities, i.e. $f_0$ is chosen with probability $p$ and $f_1$ is chosen with probability $1-p$. One of the results is that for $f$, there is a unique ergodic stationary measure $\mu_f$  with $\mu_f(\{0\} \cup \{1\}) = 0$. 

The class of systems is broadened in \cite{CzudekSzarek} to allow for random iterated function systems of the form $f = (\{f_i\}, p)$, $i = \{1, .., k\}$, where each $f_i$ is an increasing homeomorphism on the unit interval, continuously differentiable on $[0, \beta]$ and $[1-\beta, 1]$ for some fixed $\beta > 0$, and for all $x \in (0, 1)$ there exist indices $i \neq j$ such that $f_i(x) < x < f_j(x)$. The probability vector $p$ is a $k$-tuple of numbers from $(0, 1)$ such that $\sum_{i=1}^{k} p_i = 1$ and each function $f_i$ is chosen with probability $p_i$. As before, the positivity of Lyapunov exponents is still required. Although the authors of \cite{CzudekSzarek} are focused on a central limit theorem, along the way they prove the existence and uniqueness of a stationary measure $\mu$ on $(0, 1)$, as well as the additional result that $\mu$ is atomless.

The same class of systems is considered in \cite{CzernousSzarek}. At the outset, some fixed $\beta > 0$ is chosen. Then a Baire space of random iterated systems which are $C^1$-smooth on the $\beta$-neighborhoods of the endpoints is constructed by taking the closure of the subset of systems for which all $f_i$ are absolutely continuous. In this Baire space, a residual set of systems is found with a unique stationary measure on $(0, 1)$ and this measure is singular.

In one sense, we work with larger class of systems by dropping the requirements of piecewise linearity of \cite{BaranskiSpiewak}, $C^2$-smoothness  of \cite{GharaeiHomburg}, and continuous differentiability in some fixed $\beta$-neighborhoods of the endpoints of \cite{CzernousSzarek} and \cite{CzudekSzarek}. We assume only differentiability at the endpoints, which is needed to test for positive Lyapunov exponents. For this reason, our construction of a Baire space is different. On the other hand, we restrict our considerations to systems consisting of only two increasing homeomorphisms, one below and the other above the diagonal.

The tool to construct the Baire space is a conjugacy inspired by \cite{AlsedaMisiurewicz} which carries our systems from the \emph{open} interval $(0,1)$ (where all the interesting dynamics happen) to $\mathbb{R}$ in such a way that the derivatives at 0, 1 correspond to limiting distances from the diagonal at $\pm\infty$, see Lemma~\ref{Lemma:Limits_equal_to_log_of_derivatives}. Using the supremum distance from $\mathbb{R}$ we are able to define a complete metric on a suitable space of interval maps even though we only require differentiability at the two endpoints\footnote{Note that differentiability at a point is not preserved under uniform limits, so we cannot use the usual supremum distance on $[0,1]$ for our space of functions.}, see Definition~\ref{Def:CI}. This allows us to identify a Baire space structure, see Definition~\ref{Def:RDS} and Theorem~\ref{Th:RDSI_is_a_Baire_space}.

Continuing to work in $\mathbb{R}$ we find dense (but not a priori $G_\delta$) subsets in our space of random dynamical systems for which the stationary measure $\mu$ identified in \cite{CzudekSzarek} is necessarily singular (Corollary~\ref{Cor:Dense_set_with_singular_measure}) or has full support (Theorem~\ref{Th:Dense_set_with_full_support}). Key is a construction of perturbations of iterated function systems for which the stationary measure is supported on a countable union of Cantor sets (Construction~\ref{Rem:Construction_of_g}). In Section~\ref{sec:continuity} we adapt a proof from \cite{CzernousSzarek} to show that in the chosen topology, the stationary measure $\mu$ depends continuously on the random dynamical system. Then it follows from a result of Lenz and Stollmann \cite{LenzStollmann} that some of the dense sets identified in Section~\ref{sec:dense_subsets} are automatically $G_\delta$ sets. This allows us to conclude in Theorem~\ref{Th:Main_result} that a \emph{typical random dynamical system in our Baire space has unique stationary measure which is singular, non-atomic and has full support}. 

\section{Space of random dynamical systems}

We use a homeomorphism from \cite{AlsedaMisiurewicz} to define an unbounded metric on the open unit interval equivalent to the Euclidean one. 

\begin{definition}\label{Def:h}
Let $h: (0, 1) \to \mathbb{R}$ be defined by
$$h(x) = \begin{cases}
\log(x) - \log\left(\frac{1}{2}\right), & x \in \left(0, \frac{1}{2}\right], \\
\log\left(\frac{1}{2}\right) - \log(1-x), & x \in \left[\frac{1}{2}, 1\right).
         \end{cases}$$
The mapping $h$ induces metric $d_h$ on $(0, 1)$ by 
$$d_h(x, y) = |h(x) - h(y)| = \begin{cases}
  |\log(x) - \log(y)|,    & x, y \in \left(0, \frac{1}{2}\right], \\
  |\log(1-x) - \log(1-y)|, & x, y \in \left[\frac{1}{2}, 1\right),
  \end{cases}$$
and $d_h(x, y) = d_h\left(x, \frac{1}{2}\right) + d_h\left(y, \frac{1}{2}\right)$ for any other case. 
\end{definition}

Let $d(x, y) = |x - y|$ denote the Euclidean metric on $(0, 1)$. Then the reader can easily verify the following inequality
\begin{equation}\label{EQ:property_of_d_h}
d(x,y) \leq d_h(x, y),\quad \forall x, y \in (0, 1).
\end{equation}

\begin{definition}\label{Def:CI}
Let $\CI$ be the set of continuous functions $f$ from the unit interval $I=[0,1]$ onto itself such that:\begin{enumerate}[(i)]
  \item $f(0) = 0$ and $f(1) = 1$,
  \item $f((0, 1)) \subset (0, 1)$,
  \item $f'(0), f'(1)$ both exist and are strictly positive, that is $f'(0), f'(1) \in (0, \infty)$.\end{enumerate}
Define the distance between two functions from $\CI$ as the supremum distance on the open interval $(0,1)$ using the (unbounded) metric $d_h$.
$$d_{\CI}(f, g) = \sup_{x \in (0, 1)} \{ d_h(f(x), g(x))\}.$$
\end{definition}

We will show in Corollary~\ref{Cor:CI_is_complete} that this distance is always finite.

\begin{definition}\label{Def:RDS}
Let $\RDSI$ be the set of triples $f = (f_0, f_1, p_f) \in \CI \times \CI \times [0, 1]$ such that the following conditions hold:
\begin{enumerate}[(i)]
  \item $f_0$ and $f_1$ are homeomorphisms,
  \item $f_0(x) < x,\ \forall x \in (0, 1)$ (below diagonal condition),
  \item $f_1(x) > x,\ \forall x \in (0, 1)$ (above diagonal condition),
  \item $p_f \in (0, 1)$,
  \item Lyapunov exponents are positive, that is for $j \in \{0, 1\}$
  $$\Lambda_j = p_f \log f'_0(j) + (1-p_f) \log f'_1(j) > 0.$$ 
\end{enumerate}
Equip $\RDSI$ with the maximal metric
\begin{equation}\label{eq:d_m} 
d_m(f, g) = \max \{ d_{\CI}(f_0, g_0), d_{\CI}(f_1, g_1), |p_f-p_g| \}.
\end{equation}
\end{definition}

Condition (i) implies that $f_0$ and $f_1$ are strictly increasing functions since $f_0 \in \CI$ imposes the conditions $f_0(0) = 0$ and $f_0(1) = 1$, respectively for $f_1$. The space $\CI$ only admits functions with strictly positive finite derivatives at the endpoints and therefore absolute values in the definition of the Lyapunov exponents are omitted. It will be useful to make the following agreement: whenever we say $f \in \RDSI$, we consider $f$ to be a triple $(f_0, f_1, p_f)$.

We will use the homeomorphism $h$ to transfer functions between the interval and the real line. Lowercase letters will denote functions on $I$, while the corresponding uppercase letters will denote their respective images on $\mathbb{R}$. Let $\CR$ be the set of continuous functions $F$ from the real line onto itself such that $\lim_{x \to -\infty} (F(x) - x)$ and $\lim_{x \to \infty} (F(x) - x)$ both exist and are finite.

\begin{lemma}\label{Lemma:Limits_equal_to_log_of_derivatives}
The correspondence $F = \phi(f) = h \circ f\left|_{(0 ,1)}\right. \circ h^{-1}$ defines a bijection $\phi: \CI \to \CR$ and 
\begin{align}
\lim_{x \to -\infty}(F(x) - x) &= \log(f'(0)), \label{Eq:limits_to_derivatives_at0}\\
\lim_{x \to \infty}(F(x)-x) &= -\log(f'(1)).\label{Eq:limits_to_derivatives_at1}
\end{align}
\end{lemma}

\begin{proof}
\underline{Step 1}: 
Let $f \in \CI$ and $F: \mathbb{R} \to \mathbb{R}$ be defined by $F = h \circ f\left|_{(0, 1)}\right. \circ h^{-1}$. We will prove that $F \in \CR$ and that equations~(\ref{Eq:limits_to_derivatives_at0}) and (\ref{Eq:limits_to_derivatives_at1}) hold.

The function $F$ is clearly continuous. For the derivative at zero, let 
$$a = f'(0) = \lim_{x \to 0^{+}} \frac{f(x)}{x}.$$
For given $\epsilon > 0$ there exists $\delta > 0$ such that for all $x \in (0, \delta)$ it holds that $|f(x) - ax| < \epsilon x$ and therefore
\begin{equation}\label{Eq:epsilon_f_x}
(a-\epsilon)x < f(x) < (a+\epsilon)x.
\end{equation}
The linear function $x \mapsto lx$ on $I$ near zero when translated to $\mathbb{R}$ becomes the function
$$L(x) = \log \left( l \frac{e^x}{2} \right) - \log \left( \frac{1}{2} \right) = x + \log(l),$$
hence condition (\ref{Eq:epsilon_f_x}) translates to $\CR$ as
$$\log(a-\epsilon) + x < F(x) < \log(a+\epsilon) + x,$$
$$\log(a-\epsilon) < F(x) - x < \log(a+\epsilon),$$
and by making $\epsilon$ arbitrarily small, we have $F(x) - x \to \log(a)$ as $x \to -\infty$.

For the derivative at the other endpoint, the argument is similar.

\underline{Step 2}: 
Conversely, let $F \in \CR$ and $f: [0, 1] \to [0, 1]$ be defined by 
\begin{equation}\label{Eq:definition_of_psi}
f(x) = \begin{cases}
0, & x = 0, \\
h^{-1} \circ F \circ h, & x \in (0, 1), \\
1, & x = 1.
\end{cases}
\end{equation}
The continuity of $f$ is clear including an easy verification at the endpoints. We will prove that $f \in \CI$ and that $f'(0) = e^{\lim_{x \to -\infty}(F(x) - x)}$. 

Let $\lim_{x \to -\infty}(F(x) - x) = \log(l)$ that is $l = e^{\lim_{x \to -\infty}(F(x) - x)}$. For given $\epsilon > 0$ there exists $R < 0$ such that $\forall x < R$ the following hold
  $$-\epsilon < (F(x) - x) - \log(l) < \epsilon,$$
  $$-\epsilon + x + \log(l) < F(x) < \epsilon + x + \log(l),$$
  and by transferring the above to the space $\CI$ we get for all $x < h^{-1}(R)$,
    \begin{gather*}
e^{\log\left(\frac{1}{2}\right) - \epsilon + \log(x) - \log\left(\frac{1}{2}\right) + \log(l)} < f(x) < e^{\log\left(\frac{1}{2}\right) + \epsilon + \log(x) - \log\left(\frac{1}{2}\right) + \log(l)},\\
  l e^{-\epsilon} < \frac{f(x)}{x} < l e^{\epsilon},
  \end{gather*}
  therefore for the derivative of $f$ at zero as $\epsilon \to 0$ we get 
  $$f'(0) = \lim_{x \to 0^{+}} \frac{f(x)}{x} = l = e^{\lim_{x \to -\infty}(F(x) - x)}.$$
A simple repetition of the argument shows that $f'(1)=e^{-\lim_{x \to \infty}(F(x) - x)}$.

\underline{Step 3}: The operator described in step 2 serves as an inverse to $\phi$, so $\phi$ is a bijection.
\end{proof}

Recall that a topological space is called a \emph{Baire space} if the intersection of any countable family of open dense subsets is dense. The following facts are used: by the \emph{Baire Category Theorem}, a complete metric space is a Baire space in the topology induced by the metric, and any $G_{\delta}$ subset (in particular any closed subset) of a complete metric space is a Baire space in the subspace topology.

\begin{corollary}\label{Cor:CI_is_complete}
The space $(\CI, d_{\CI})$ is a complete metric space.
\end{corollary}

\begin{proof}
It is a known fact that the space of continuous bounded real-valued functions on $\mathbb{R}$ equipped with the supremum metric is a complete metric space. The functions which have finite limits at both $\pm\infty$ form a closed subspace there, hence are again a complete metric space. But $\CR$ was just defined as those functions for which $F(x)-x$ has finite limits at both $\pm\infty$. The operator $F\mapsto (F-\id)$ is an isometry, so $\CR$ is a complete metric space in the supremum metric. We showed in Lemma~\ref{Lemma:Limits_equal_to_log_of_derivatives} that $\phi:\CI\to\CR$ is a bijection, and recalling Definition~\ref{Def:CI} we have
$$d_{\CI}(f, g) = \sup_{x \in (0, 1)} \{ d_h(f(x), g(x))\} = \sup_{x\in\mathbb{R}} |F(x)-G(x)|,$$
where $F, G$ are the respective images of $f, g$ under $\phi$. This shows that $d_{\CI}$ is a well-defined metric (it is finite), and that $\phi$ is an isometry.
\end{proof}

\begin{corollary}\label{Cor:xi_is_continuous_function}
The mappings $\xi_0, \xi_1: \CI \to (0, \infty)$ defined by $\xi_0(f) = f'(0)$ and $\xi_1(f) = f'(1)$ are continuous.
\end{corollary}
\begin{proof}
Using Lemma~\ref{Lemma:Limits_equal_to_log_of_derivatives} it follows that given $f \in \CI$, for every $g$ in the $\epsilon$-ball $B_{\epsilon}(f)$ we have
$g'(0) \in (e^{-\epsilon}f'(0), e^{\epsilon}f'(0))$ and $g'(1) \in (e^{-\epsilon}f'(1), e^{\epsilon}f'(1)).$
\end{proof}

The above two corollaries are the very reason for choosing the specific homeomorphism $h$ between the real line and the open unit interval rather then some other homeomorphism. The resulting metric $d_{\CI}$ can ``see" the derivatives at the endpoints.

\begin{theorem}\label{Th:RDSI_is_a_Baire_space}
The space $(\RDSI, d_m)$ is a Baire space in the topology induced by the metric.
\end{theorem}

\begin{proof}
Let $\XI = \CI \times \CI \times [0, 1]$. This is a Cartesian product of finitely many complete metric spaces and hence a complete metric space in the maximal metric defined by Equation~(\ref{eq:d_m}) for $f, g \in \XI$.

We will show that each condition in Definition~\ref{Def:RDS} forms a $G_{\delta}$ subset of $\XI$ and then $\RDSI$, being the intersection of a finite number of $G_{\delta}$ subsets, is therefore a $G_{\delta}$ subset of $\XI$.
\begin{enumerate}[(i)]
  \item Let $Y \subset \CI$ be the subset of all homeomorphisms. The condition $f \in Y$ is equivalent to the condition
  \begin{align*}
  f(x_1) < f(x_2)&,\quad \forall x_1 \in [0, 1] \cap \mathbb{Q}, \forall x_2 \in (x_1, 1] \cap \mathbb{Q},
\end{align*}
and we can now express the subset $Y$ as a countable intersection of open sets
$$Y = \bigcap_{x_1 \in [0, 1] \cap \mathbb{Q}} \ \bigcap_{x_2 \in (x_1, 1] \cap \mathbb{Q}}  \{ f \in \CI : f(x_1) < f(x_2)\},$$
hence $Y \times \CI \times [0, 1]$ is a $G_{\delta}$ subset of $\XI$. A similar argument for $f_1$ shows that $\CI \times Y \times [0, 1]$ is also a $G_{\delta}$ subset of $\XI$.
  \item For $n \in \mathbb{N}, n > 1$, let $I_n = \left[\frac{1}{n}, 1-\frac{1}{n}\right]$ and $U_n = \left\{ f \in \CI : f(x) < x,\ \forall x \in I_n \right\}$. The set $U_n$ is open since if $f \in U_n$ then by continuity of $f$ and compactness of $I_n$ there exists $\delta > 0$ such that $f(x) + \delta < x,\ \forall x \in I_n$. If $g \in B_{\delta}(f)$ then by equation~\eqref{EQ:property_of_d_h}, $\forall x \in I_n$ we have $|f(x) - g(x)| \leq d_{h}(f(x), g(x)) < \delta$ which implies  $g(x) < f(x) + \delta < x$ on $I_n$, hence $U_n$ is open. The set $U = \bigcap^{\infty}_{n = 2} U_n$ is the set of all functions satisfying the condition of being under the diagonal on $(0, 1)$, hence $U$ is a $G_{\delta}$ set and subsequently $U \times \CI \times [0, 1]$ is $G_{\delta}$. 
  \item The proof is analogous to (ii).
  \item The set $\CI \times \CI \times (0, 1)$ is a $G_{\delta}$ set since it is open.
  \item It follows from Corollary~\ref{Cor:xi_is_continuous_function} that the Lyapunov exponents as functions $\Lambda_0, \Lambda_1: \XI \to \mathbb{R}$ are continuous. Therefore the set $\Lambda^{-1}_0( (0, \infty) )$, resp. $\Lambda^{-1}_1( (0, \infty) )$, is open, hence $G_{\delta}$. \qedhere
\end{enumerate}
\end{proof}

We conclude our discussion about Baire spaces of random dynamical systems by going back to the real line. A random dynamical system $f=(f_0,f_1,p_f)\in\RDSI$ can be transferred to the real line by taking the image of both $f_0,f_1$ under $\phi$ and leaving the probability unchanged. We overload the symbol $\phi$ and write $F=(F_0,F_1,p_F) := \phi(f)=(\phi(f_0),\phi(f_1),p_f)$. Let $\RDSR$ be the space of all systems $F$ obtained in this way. The conditions in Definition~\ref{Def:RDS} translate in the following way:
\begin{enumerate}[(i)]
  \item $F_0$ and $F_1$ are homeomorphisms,
  \item $F_0(x) < x, \forall x \in \mathbb{R}$ (below diagonal condition),
  \item $F_1(x) > x, \forall x \in \mathbb{R}$ (above diagonal condition),
  \item $p_F \in (0, 1)$,
  \item Lyapunov exponents at $\pm\infty$ are positive, i.e.
\begin{equation}\label{eq:Lyapunov_exponent_at_inft}
\begin{aligned}
 \Lambda_{-\infty} &= p_F \lim_{x \to -\infty} (F_0(x) - x)) + (1-p_F) \lim_{x \to -\infty} (F_1(x) - x) > 0, \\
 \Lambda_{+\infty}  &= p_F (-\lim_{x \to \infty} (F_0(x) - x))) + (1-p_F)(- \lim_{x \to \infty} (F_1(x) - x)) > 0.
\end{aligned}
\end{equation}
\end{enumerate}
Note that at plus infinity the limits take the opposite sign, see Lemma~\ref{Lemma:Limits_equal_to_log_of_derivatives}.


\begin{remark}\label{Rem:admissible_system_is_in_RDSI}
The pairs of homeomorphisms allowed as \emph{admissible iterated function systems} in \cite[Definitions 1 \& 2]{CzernousSzarek} and also \cite[Definitions 1 \& 2]{CzudekSzarek} are also allowed in $\RDSI$, but not conversely, since we do not require smoothness on a neighborhood of the endpoints. Nevertheless, some of the results and ideas from these articles are applicable in our setting and we use them in the sequel.
\end{remark}
\section{Markov operator, stationary measure and other notions}

In this section, we recall notions needed in the remainder of the text. Let $U$ be an interval in $\mathbb{R}$ and denote by $\Mp(U)$, resp. by $\Mfin(U)$, the set of all probability measures, resp. the set of all finite measures, on the $\sigma$-algebra of Borel sets $\B(U)$. 

\begin{definition} 
An operator $P : \Mfin(U) \to \Mfin(U)$ is called a \emph{Markov operator} if it is positive linear and preserves the measure, that is
\begin{enumerate}[(i)]
  \item $P(\lambda_1 \mu_1 + \lambda_2 \mu_2) = \lambda_1 P\mu_1 + \lambda_2 P\mu_2,\ \lambda_1, \lambda_2 > 0,\ \mu_1, \mu_2 \in \Mfin(U)$,
  \item $P\mu(U) = \mu(U),\ \mu \in \Mfin(U)$.
\end{enumerate}
\end{definition}

A general random dynamical system $(\{f_i\}, \{p_i\}),\ i=1, ..., k$, where $\{f_i\}$ is a $k$-tuple of maps from $U$ to itself and $\{p_i\}$ is a probability vector, generates a Markov operator $P: \Mfin(U) \to \Mfin(U)$ of the form 
$$P\mu = \sum_{i=1}^k p_i f_i \mu,$$
where $f_i \mu (A) = \mu(f^{-1}_i(A))$ for $A \in \B(U)$, which describes the evolution of measure due to the action of the randomly chosen map. Therefore $f \in \RDSI$ generates a Markov operator $P_f: \Mfin(0, 1) \to \Mfin(0, 1)$ of the form
$$P_f\mu(A) = p_f \mu(f^{-1}_0(A)) + (1-p_f) \mu(f^{-1}_1(A)),\quad A \in \B(0, 1).$$
Similarly $F \in \RDSR$ generates the corresponding Markov operator
$$P_F\mu(A) = p_F \mu(F^{-1}_0(A)) + (1-p_F) \mu(F^{-1}_1(A)),\quad A \in \B(\mathbb{R}).$$

\begin{definition}
A measure $\mu \in \Mfin(0,1)$ is called \emph{stationary} for $f\in\RDSI$ if $P_f\mu = \mu$. Similarly a measure $\mu \in \Mfin(\mathbb{R})$ is called \emph{stationary} for $F\in\RDSR$ if $P_F\mu=\mu$.
\end{definition}

\begin{theorem}[{\cite[Theorem 1]{CzudekSzarek}}]\label{Th:Existence_of_unique_stationary_measure}
A random dynamical system $f \in \RDSI$ has a unique stationary measure $\mu_f \in \Mp(0, 1)$. Moreover $\mu_f$ is atomless.
\end{theorem}

\begin{proof}
We noted in Remark~\ref{Rem:admissible_system_is_in_RDSI} that $f$ satisfies the requirements of an admissible iterated function system from \cite{CzudekSzarek} save continuous differentiability on some neighborhood of the endpoints. This property is not used in the proof of \cite[Theorem~1]{CzudekSzarek} and therefore the same proof is valid also for~$f \in \RDSI$.
\end{proof}

The homeomorphism $h$ defines a pushforward operator $h_{*}: \Mfin(0, 1) \to \Mfin(\mathbb{R})$ which carries $\mu \in \Mfin(0, 1)$ to
$$h_{*}\mu(A) = \mu(h^{-1}(A)),\ A \in \B(\mathbb{R}),$$
and preserves the property of being stationary. Therefore whenever we work with $f \in \RDSI$ and its stationary measure $\mu_f$, we can ``jump'' to $\RDSR$ by setting $F = \phi(f)$ and then $\mu_F = h_{*}\mu_f$ is its unique stationary measure on the real line.

The \emph{support of $\mu \in \Mp(U)$}, denoted $\supp(\mu)$, is defined as the set of all points $x \in U$ for which every open neighborhood $N_x$ of $x$ has positive measure. Note that $\supp(\mu)$ is automatically a closed set in $U$.
A measure $\mu \in \Mp(U)$ has \emph{full support} if $\mu(A) > 0$ for every open subset $A$ of $U$. \emph{An atom} is a singleton set with positive measure and a measure which has no atoms is called \emph{atomless}. A measure $\mu \in \Mp(U)$ is called \emph{singular} if there exists a set $A \in \B(U)$ such that $\mu(A)=1$ while the Lebesgue measure of $A$ is zero. Let $\mu_n$ be a sequence of probability measures from $\Mp(U)$. We define \emph{weak$^*$ convergence} of probability measures $\mu_n$ to $\mu \in \Mp(U)$, denoted by $\mu_n \xrightarrow{\text{w}^*} \mu$,  if and only if $\langle \mu_n, \xi\rangle \to \langle \mu, \xi\rangle$ for every continuous bounded function $\xi$ on $U$, where $\langle \mu, \xi\rangle = \int_U \xi\ d\mu$. A collection $M$ of probability measures from $\Mp(U)$ is called \emph{tight} if, for any $\epsilon > 0$, there is a compact subset $K_{\epsilon}$ of $U$ such that for all measures $\mu \in M$
$$\mu(K_{\epsilon}) > 1 - \epsilon.$$

Finally, we give a few definitions from the topological theory of random dynamical systems. For a system $f \in \RDSI$, and a finite word $\omega=i_0 i_1 \cdots i_{n-1} \in \{0,1\}^n$, $n \in \mathbb{N}$, let $f_{\omega}$ be the composition $f_{i_{n-1}} \circ \cdots \circ f_{i_1} \circ f_{i_0}$. The \emph{orbit of a point} $x \in [0, 1]$ under $f$ is $\Orb(x) = \{ f_{\omega}(x)~:~\omega \in \{0,1\}^n, n \in \mathbb{N} \}$. The set $A \subset I$ is \emph{invariant} under $f$ if $\Orb(x) \subset A$ for every $x \in A$. We use similar notation for $F\in\RDSR$ as well. A random dynamical system $F \in \RDSR$ is \emph{minimal} if every orbit under $F$ is dense in $\mathbb{R}$. Note that systems from $\RDSI$ are never minimal because there are fixed points at the endpoints.

\section{Dense subsets of $\RDSI$}\label{sec:dense_subsets}

\subsection{Systems where the support of the stationary measure is a null set}
\strut \\ \indent
Let $\Leb(U)$ denote the Lebesgue measure of any Borel set $U \subset \mathbb{R}$.

\begin{theorem}\label{Th:Dense_set_with_Lebesgue_support_zero}
The systems $f \in \RDSI$ for which $\Leb(\supp(\mu_f)) = 0$ form a dense subset of $\RDSI$.
\end{theorem}

\begin{corollary}\label{Cor:Dense_set_with_singular_measure}
For a dense set of systems $f\in\RDSI$, the stationary measure $\mu_f$ is singular with respect to the Lebesgue measure.
\end{corollary}

First, we outline the proof of Theorem~\ref{Th:Dense_set_with_Lebesgue_support_zero}. For arbitrary $f \in \RDSI$ and given $\epsilon > 0$, we construct $g = \RDSI$ which is $\epsilon$-close to $f$. The system $g$ is constructed in such a way that the support of its stationary measure is a countable union of adjacent Cantor sets, each of them with Lebesgue measure zero, union with the set $\{ 0, 1 \}$. In doing so, we will take advantage of the fact that the space $\RDSR$ is isometric to $\RDSI$, transfer $f \in \RDSI$ to $F \in \RDSR$, construct our $\epsilon$-close system $G \in \RDSR$ and transfer $G$ back to $g \in \RDSI$. In the process of this construction, we also make sure that conditions $(i)$ to $(v)$ from Definition \ref{Def:RDS} are satisfied for $g$ to belong to $\RDSI$. 

The following lemma will be used later in the proof and is a standard fact from analysis.
\begin{lemma}\label{Lemma:OrderHomeomorphismExtension}
Let $C_1, C_2 \subset \mathbb{R}$ be two Cantor sets and denote $m_{i} = \min\{C_i\}$, $M_{i} = \max\{C_i\}$ for $i \in \{ 1, 2\}$. Then there exists an order-preserving homeomorphism ${f: [m_{1}, M_{1}] \to [m_{2}, M_{2}]}$ such that $f(C_1) = C_2$.
\end{lemma}

\begin{construction}\label{Rem:Construction_of_g}
Let $f = (f_0, f_1, p) \in \RDSI$ and $\epsilon > 0$ be given.

\underline{Step 1:} Let $F = \phi(f) \in \RDSR$. Recall that the values of the limits $F_0(x)-x$, $F_1(x)-x$ at both infinities are given by Lemma~\ref{Lemma:Limits_equal_to_log_of_derivatives}. It follows that we can approximate $F_0, F_1$ near $\pm\infty$ by translations. Choose $R>0$ large enough so that 
\begin{align*}
|(F_0(x) - x) - \log(f'_0(0))| &< \frac{\epsilon}{2},\quad |(F_1(x) - x) - \log(f'_1(0))| < \frac{\epsilon}{2},
\end{align*}
whenever $\min(x,F_0(x),F_1(x))<-R$ and
\begin{align*}
|(F_0(x) - x) + \log(f'_0(1))| &< \frac{\epsilon}{2},\quad |(F_1(x) - x) + \log(f'_1(1))| < \frac{\epsilon}{2},
\end{align*}
whenever $\max(x,F_0(x),F_1(x))>R$.

Approximate $F_0$ on the interval $(-\infty, R]$ by an affine function of the form $A(x) = x + a$ where $a \in (\log(f'_0(0)) - \epsilon/2, \log(f'_0(0)) + \epsilon/2)$ and similarly on the interval $[R, \infty)$ approximate $F_0$ by an affine function of the form $B(x) = x + b$ where $b \in (-\log(f'_0(1)) - \epsilon/2, -\log(f'_0(1)) + \epsilon/2)$ and similarly for $F_1$, see Figure~\ref{Fig:R_box}. 

Additionally assume that $R$ is large enough that $F_0(R)>-R$ and $F_1(-R)<R$, which guarantees that the graphs of the functions ``pass through'' the $R$-box. 

It remains to choose appropriate linear functions on the complement of the interval $[-R, R]$, approximate $F_0$ and $F_1$ on the interval $[-R, R]$ and ``connect the bits."
 
\begin{figure}[h!]
  \centering
  \begin{subfigure}{0.5\textwidth}
  \centering
  \includegraphics[width=6cm]{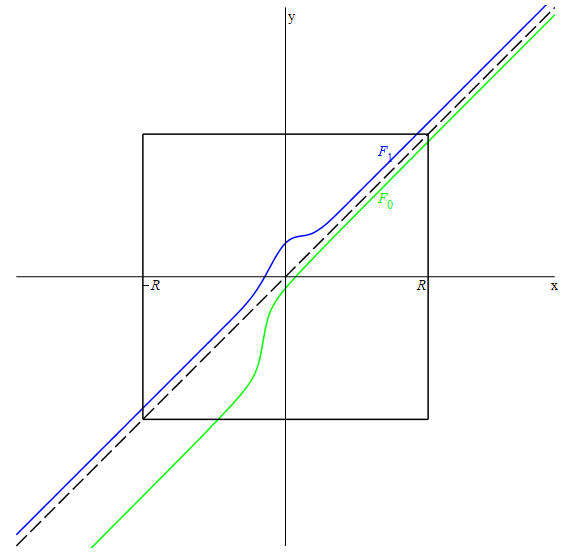}
  \caption{R-box}
  \label{Fig:R_box}
  \end{subfigure}%
  \begin{subfigure}{0.5\textwidth}
  \centering
  \includegraphics[width=6cm]{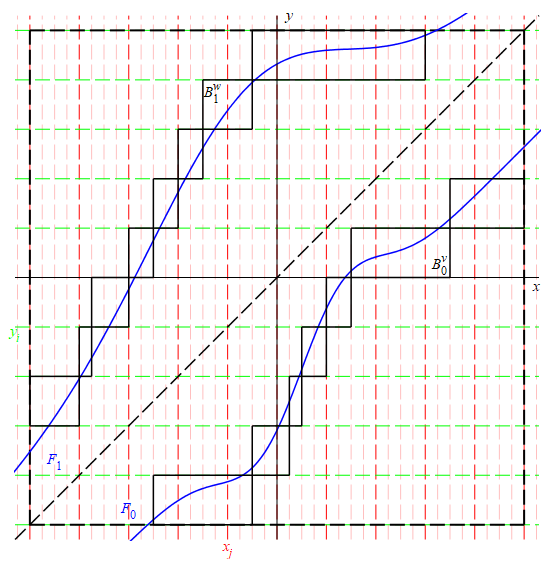}
  \caption{Linked boxes}
  \label{Fig:M_levels}
  \end{subfigure}
  \caption{Construction of $\epsilon$-close system}
\end{figure}

\underline{Step 2:} Without loss of generality, we can assume $R = \frac{1}{2}$ for we can scale all our calculations using the linear transformation $x \mapsto \frac{1}{2R}x$ with the inverse $x \mapsto 2Rx$. Henceforward, we proceed on the interval $\left[-\frac{1}{2}, \frac{1}{2}\right]$.

\underline{Step 3:} Choose $M \in \mathbb{N}$ such that all three conditions below are satisfied. The purpose of each condition will be explained at an appropriate stage of the construction.
\begin{align}
\label{ConditionM-1} &\frac{1}{2^M} < \frac{\epsilon}{2},\\
\label{ConditionM-2} &\frac{1}{2^M} < \min_{x \in \left[-\frac{1}{2}, \frac{1}{2} \right]} \{ |F_0(x) - x|, |F_1(x) - x| \},\\
\label{ConditionM-3} &\frac{1}{2^M} < \min\left\{\left|-\tfrac{1}{2} - F_0\left(\tfrac{1}{2}\right)\right|, \left|\tfrac{1}{2} - F_1\left(-\tfrac{1}{2}\right)\right|\right\}.
\end{align}
Define a partition of the interval $\left[-\frac{1}{2}, \frac{1}{2}\right]$ on the $y$-axis by points
$$y_i = \frac{-2^{M-1}+i}{2^M},\quad i = 0, 1, ..., 2^M,$$
evenly spaced by $\frac{1}{2^M}$. The intervals $[y_{i-1}, y_i]$, $i = 1, ..., 2^M$ form steps by which we approximate the functions $F_0$ and $F_1$ at a later stage.

\underline{Step 4:} Both $F_0$ and $F_1$ are uniformly continuous on the interval $\left[-\frac{1}{2}, \frac{1}{2}\right]$. Therefore for the given number $\frac{1}{2^M}$, there exists $\delta_0$, resp. $\delta_1$, such that for arbitrary $x, y \in \left[-\frac{1}{2}, \frac{1}{2}\right]$, $|x-y| < \delta_0$ we have $|F_0(x) - F_0(y)| < \frac{1}{2^M}$, resp. $|x-y| < \delta_1$ we have $|F_1(x) - F_1(y)| < \frac{1}{2^M}$. Choose $M' \in \mathbb{N}$ such that
\begin{equation*}
M' > M \text{ and }
\frac{1}{2^{M'}} < \min\{ \delta_0, \delta_1 \},
\end{equation*}
and define a finer partition of the interval $\left[-\frac{1}{2}, \frac{1}{2}\right]$ on the $x$-axis by points 
$$x_j = \frac{-2^{M' - 1} + j}{2^{M'}},\quad j = 0, 1, ..., 2^{M'},$$
evenly spaced by $\frac{1}{2^{M'}}$. In essence, every time the function $F_0$, resp. $F_1$, increases by $\frac{1}{2^M}$ there exists a point $x_j$, resp. $x_k$, such that $y_{i-1} < F_0(x_j) < y_i$, resp. $y_{i-1} < F_1(x_k) < y_i$.

\underline{Step 5:} Define a set $S$ on the real line as 
$$S = \bigcup_{z \in \mathbb{Z}} \left( C_{m} + \frac{z}{2^{M'}}\right),$$
where $C_{m}$ is the middle-thirds Cantor set constructed on the interval $[x_0, x_1]$. We simply cut the whole real line into closed adjacent  intervals of length $\frac{1}{2^{M'}}$ and construct middle-thirds Cantor sets on each of them. Note that every point $y_i$, resp $x_j$, coincides with endpoints of two adjacent middle-thirds Cantor segments of $S$. The set $S$ is not a Cantor set since we lost compactness but its intersection with any finite union of intervals of the form $[y_{i-1}, y_i]$, resp. $[x_{j-1}, x_j]$, is a Cantor set. We choose the system $G$ in such a way that $S$ is invariant under $G$.

\underline{Step 6:} Construct ``linked boxes" for the functions $F_0$ and $F_1$. The construction can be described informally by the following rules.\\
Rules for the function $F_0$:
\begin{enumerate}[(i)]
  \item Follow the function from the right to the left.
  \item Construct each box so that the function enters on the first available occasion from the top.
  \item Let each box have height $\frac{1}{2^M}$ so that the function decreases by only one level on each box.
  \item Let the function leave the box on the last available occasion on the left.
\end{enumerate}
Rules for the function $F_1$:
\begin{enumerate}[(i)]
  \item Follow the function from the left to the right.
  \item Construct each box so that the function enters on the first available occasion from the bottom.
  \item Let each box have height $\frac{1}{2^M}$ so that the function increases by only one level on each box.
  \item Let the function leave the box on the last available occasion on the right.
\end{enumerate}
Figure~\ref{Fig:M_levels} illustrates the procedure which can be formally described in the following way for $F_0$.
\begin{enumerate}[(i)]
  \item Let $j = 2^{M'}$ and $v=1$.
  \item Let $r_x = x_{j}$ and $i = \max\{ 0, 1, ..., 2^M : y_i \leq F_0(r_x) \}$. Set $r_y = y_i$ and $r_p = (r_x, r_y)$. Also set $k = i$, which is the number of boxes being constructed.
  \item Let $l_y = y_{i-1}$ and $j = \min\{ 0, 1, ..., 2^{M'} : F_0(x_j) \geq l_y\}$. Then $l_x = x_j$ and $l_p = (l_x, l_y)$.
  \item The points $l_p$ and $r_p$ form the box $B^v_{0} = X^v_{0} \times Y^v_{0} = [l_x, r_x] \times [l_y, r_y]$.
  \item If $i = 1$ then we are done. Otherwise set $v = v + 1$, $i = i - 1$, $r_p = l_p$, and jump to step (iii). 
\end{enumerate}
For the function $F_0$, the boxes $B^v_0 = X^v_0 \times Y^v_0,\ v = 1, ..., k$ have the following properties. Condition~\eqref{ConditionM-3} guarantees that at least one box was constructed. Any function defined on any box $B^v_0$, that is $G^v_0: X^v_0 \to Y^v_0$, is $\epsilon$-close to the restriction of the function $F_0$ to $X^v_0$, which is guaranteed by condition~\eqref{ConditionM-1} for the selection of $M$. Condition~\eqref{ConditionM-2} guarantees that no box $B^v_0$ intersects the diagonal, hence the respective function $G^v_0$ is always below the diagonal. Additionally, step 4 guarantees that every box has positive width.

\underline{Step 7:} For each box $B^v_0$, by Lemma~\ref{Lemma:OrderHomeomorphismExtension} there is an order-preserving homeomorphism $G^v_0: X^v_0 \to Y^v_0$ such that $G_0^v(X_0^v \cap S) = Y_0^v \cap S$. Let $G^v_0$ equal zero on the rest of the real line.

\underline{Step 8:} Finally, we are ready to construct the function $G_0$. Let $X_0 = \bigcup_{v=1}^{k} X^v_0$, $Y_0 = \bigcup_{v=1}^{k} Y^v_0$,  
\begin{align*}
L_x &= \min \{ X_0 \}, & L_y &= \min \{ Y_0 \}, \\
R_x &= \max \{ X_0 \}, & R_y &= \max \{ Y_0 \},
\end{align*}
that is $(L_x, L_y)$ is the point where our function under construction enters, resp. $(R_x, R_y)$ is the point where it leaves, the large box $\left[-\frac{1}{2}, \frac{1}{2}\right] \times \left[-\frac{1}{2}, \frac{1}{2}\right]$. Then $G_0: \mathbb{R} \to \mathbb{R}$ defined by
$$ G_0(x) = \begin{cases} 
      x + L_y,           & x \in (-\infty, L_x) \\
      \sum_v{G^v_0}(x),  & x \in [L_x, R_x], \\
      x + R_y,           & x \in (R_x, \infty)
   \end{cases}
$$
is $\epsilon$-close to $F_0$.

\underline{Step 9:} A similar procedure can be formalized using the informal description above for $F_1$, including the considerations to construct the function $G_1$. 

\underline{Step 10:} The approximating system $G$ is then defined by $G = (G_0, G_1, p_G = p_F)$ and the final remaining step is to transfer it back to the space $\RDSI$ that is $g = \phi^{-1}(G)$. Transfer the set $S$ to the unit interval as well and append the endpoints, i.e. $s = h^{-1}(S) \cup \{0, 1 \}$. By appending the endpoints, the set $s$ becomes closed and a Cantor set again.
\end{construction}

\begin{lemma} 
The set $s$ is invariant under $g$.
\end{lemma}

\begin{proof}
We continue to use the notation from Construction~\ref{Rem:Construction_of_g}. It is sufficient to show that $S$ is invariant for $G=(G_0, G_1, p_G)$. Let $x\in S$.
\begin{enumerate}[(i)]
  \item If $x < L_x$, then $G_0(x) \in S$, because we note that $L_y$ is of the form $\displaystyle\frac{z_l}{2^{M'}}$ for some $z_l \in \mathbb{Z}$ and $x \mapsto x + L_y$ is then translation ``to the left" by a multiple of $\displaystyle\frac{1}{2^{M'}}$, see the definition of $S$;
  \item if $x > R_x$, then $G_0(x) \in S$, because again $R_y = \frac{z_r}{2^{M'}}, z_r \in \mathbb{Z}$;
  \item if $x \in [L_x, R_x]$, then $G_0(x) \in S$ by the definition of $G_0$.
\end{enumerate}
A similar analysis applies to $G_1$, hence $G_{\omega}(x) \in S$ for an arbitrary finite word $\omega$. Thus $\Orb(x)\subset S$.
\end{proof}

\begin{proof}(of Theorem \ref{Th:Dense_set_with_Lebesgue_support_zero})
Fix $f \in \RDSI$, $\epsilon>0$ and let $g$ be the system from Construction~\ref{Rem:Construction_of_g} and $s$ its invariant Cantor set. Then $d_m(f,g)<\epsilon$.
Further let $P_g$ be its corresponding Markov operator and $\mu_g$ its unique stationary measure given by Theorem~\ref{Th:Existence_of_unique_stationary_measure}. Then by \cite[Theorem~2]{CzudekSzarek}, for any $\mu \in \Mp(0, 1)$ we have
$$P_g^n\mu \xrightarrow{\text{w}^*} \mu_g.$$

We remark again that article~\cite{CzudekSzarek} uses the notion of an \emph{admissible iterated function system}, but the proof of \cite[Theorem~2]{CzudekSzarek} does not require continuous differentiability at endpoints, hence it is valid for $g$. Let  $a \in s$ be arbitrary and let $\mu$ be defined by
$$\mu(A) = \begin{cases}
  1, &a \in A, \\
  0, &a \not \in A.
           \end{cases}$$
Since $s$ is invariant, we have
$$P_g \mu(s) = p_g \mu(g^{-1}_0(s)) + (1- p_g) \mu(g^{-1}_1(s)) \geq p_g \mu(s) + (1- p_g) \mu(s) = 1,$$
therefore $P_g^n \mu (s) = 1$ for $n \in \mathbb{N}$ and clearly $\limsup\limits_{n \to \infty} P_g^n \mu (s) = 1$. By the Portmanteau Theorem, if $P_g^n\mu \xrightarrow{\text{w}^*} \mu_g$ then for the closed set $s$ we have $\limsup\limits_{n \to \infty} P_g^n \mu (s) \leq \mu_g(s)$, hence $\mu_g(s) = 1$. Since $s$ is closed this shows that $\supp(\mu_g) \subset s$. Therefore $\Leb(\supp(\mu_g))=0$.
\end{proof}

\subsection{Systems with fully supported stationary measures}
\strut \\ \indent
We can find another dense subset of $\RDSI$ using the tools from \cite{BaranskiSpiewak} for piecewise linear homeomorphisms on the unit interval. The space $\RDSI$ contains piecewise linear homeomorphisms and we utilize the so-called non-resonant case from \cite{BaranskiSpiewak} on a neighborhood of zero to approximate any given $f \in \RDSI$ to show that the systems with fully supported stationary measures form another dense subset of $\RDSI$.

\begin{theorem}\label{Th:Dense_set_with_full_support}
The systems $f \in \RDSI$ for which the unique stationary measure $\mu_f \in \Mp(0, 1)$ has full support form a dense subset of $\RDSI$.
\end{theorem}

As in \cite{BaranskiSpiewak}, the proof is based on the notion of minimality. Systems from $\RDSI$ are never minimal since the endpoint set is invariant, so we work instead on $\mathbb{R}$ where there are no endpoints.

The next lemma is standard, see eg.~\cite[Lemma 2]{Gelfert}.

\begin{lemma}\label{Lemma:minimal_system_has_measure_with_full_support}
If $F \in \RDSR$ is minimal, then its unique stationary measure $\mu_F$ has full support.
\end{lemma}


Before we proceed with the next lemma, we pause to consider Lyapunov exponents on the real line. Recall from equation~(\ref{eq:Lyapunov_exponent_at_inft}) that the Lyapunov exponent for $F \in \RDSR$ on the real line at minus infinity takes the form
$$\Lambda_{-\infty} = p_F \lim_{x \to -\infty}(F_0(x) -x ) + (1-p_F) \lim_{x \to -\infty} (F_1(x) - x) > 0.$$
Because $F_0$ is below the diagonal on the whole real line then in order to satisfy positivity of the Lyapunov exponent, the necessary condition is
$$\lim_{x \to -\infty} (F_1(x) - x) > 0.$$
The point is that $\inf_{x < 0} \{ F_1(x) - x \} > 0$. 

\begin{lemma}\label{Lemma:Minimal_systems_are_dense}
Minimal systems form a dense subset of $\RDSR$.
\end{lemma}

\begin{proof}
Let $F \in \RDSR$ and $\epsilon > 0$ be given. Let $L_i = \lim_{x \to -\infty} (F_i(x) - x)$ and decrease $\epsilon$ if necessary so that $\epsilon < L_1$ and let $R > 0$ be large enough such that
$$|(F_i(x) - x) - L_i| < \frac{\epsilon}{2},\quad\text{for}\ x \in (-\infty, -R].$$
Choose numbers $\eta_1 \in \left(L_1 - \frac{\epsilon}{2}, L_1\right)$ and $\eta_0 \in \left(L_0 - \frac{\epsilon}{2}, L_0 \right)$ such that $p_F \eta_0 + (1-p_F) \eta_1 > 0$ and $\frac{\eta_1}{\eta_0}$ is irrational and let $G_i: \mathbb{R} \to \mathbb{R}$ be
$$G_i(x) = \begin{cases}
  x + \eta_i, &x \in (-\infty, -R - \epsilon], \\
  F_i(x), &x \in [-R, \infty),
\end{cases}$$
and let $G_i$ be linear on $[-R-\epsilon, -R]$ to connect the ``left and right bits," so that the graph of $G_i$ connects the points $(-R-\epsilon, -R-\epsilon+\eta_i)$ and $(-R, F_i(-R))$, see Figure~\ref{Fig:minimal_system_G}. 
Keeping the probability unchanged, put $p_G = p_F$ and $G = (G_0, G_1, p_G)$. It is easy to see that $G \in \RDSR$ since $G_i$ are strictly increasing continuous functions, $G_0$ is below and $G_1$ above the diagonal, Lyapunov exponents are positive and also $d_m(F, G) < \epsilon$. Now fix $x \in \mathbb{R}$. Denoting the orbit of $x$ under $G$ by $\Orb(x)$, we need to show that $\overline{\Orb(x)} = \mathbb{R}$. Since $G_0$ is below the diagonal, we can replace $x$ by some iterated image $G_0^n(x)$ and assume that $x \in (-\infty, -R - \epsilon]$. It is a standard argument that for $\eta_0 < 0$, $\eta_1 > 0$, and $\frac{\eta_1}{\eta_0}$ irrational, the set of linear combinations $\{ s\eta_0 + t\eta_1 : s, t \in \mathbb{N} \}$ is dense in $\mathbb{R}$, see eg. \cite[Proposition~4.3]{BaranskiSpiewak}. Since $G_i$ restricted to $(-\infty, -R-\epsilon]$ is the translation by $\eta_i$, we have
$$G^t_1(G^s_0(x)) = x + s\eta_0 + t\eta_1,\quad\text{whenever}\ s, t \in \mathbb{N},\ s\eta_0 + t\eta_1 < 0.$$ 
This shows that the orbit of $x$ is dense in $(-\infty, x]$. Again, since $G_1$ is above the diagonal, then for any bounded open interval $U$ there exists $k \in \mathbb{N}$ such that $G^{-k}_1(U) \subset (-\infty, x]$, and therefore choosing appropriate $s, t \in \mathbb{N}$ we have $G^k_1(G^t_1(G^s_0(x))) \in U$. Therefore the orbit of an arbitrary point $x$ is dense in $\mathbb{R}$ and $G$ is minimal.
\end{proof}

\begin{figure}[h!]
  \centering
  \includegraphics[width=0.5\linewidth]{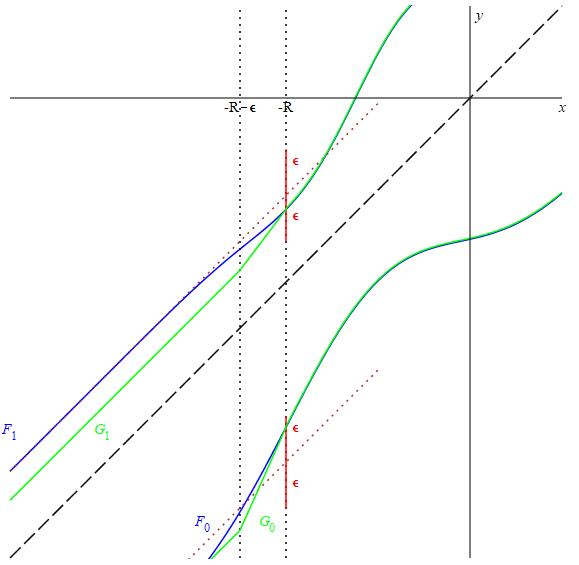}
  \caption{Minimal system $G$}
  \label{Fig:minimal_system_G}
\end{figure}

\begin{proof}(of Theorem~\ref{Th:Dense_set_with_full_support})
By Lemma~\ref{Lemma:minimal_system_has_measure_with_full_support} and Lemma~\ref{Lemma:Minimal_systems_are_dense} the systems $F \in \RDSR$ for which the unique stationary measure $\mu_F \in \Mp(\mathbb{R})$ has full support form a dense subset $Y$ in $\RDSR$. Therefore $\phi^{-1}(Y) \subset \RDSI$ is dense and the systems from $\phi^{-1}(Y)$ have fully supported stationary measures in $\Mp(0, 1)$.
\end{proof}

\section{Continuity of the stationary measure}\label{sec:continuity}

\begin{theorem}\label{Th:mu_start_is_continuous}
The operator $\psi: \RDSI \to \Mp(0, 1)$ which assigns to $f$ its unique stationary measure $\mu_f$ is continuous. 
\end{theorem}

The proof extends an argument from \cite{CzernousSzarek}. Let $\NMalpha \subset \Mp(0, 1)$ be defined by
$$\NMalpha = \{ \mu \in \Mp(0, 1) : \mu((0, x)) \leq Mx^{\alpha}\ \text{and}\ \mu((1-x, 1)) \leq Mx^{\alpha} \},$$
where $M > 0$ and $\alpha \in (0, 1)$. It follows from the regularity of measures that $\NMalpha$ is the same set as in \cite{CzernousSzarek} where closed intervals $\mu([0, x]) \leq Mx^{\alpha}$ and $\mu([1-x, 1]) \leq Mx^{\alpha}$ are used.

\begin{lemma}\label{Lemma:MNalpha_is_compact}
The set $\NMalpha$ is a compact subset of $\Mp(0, 1)$ in the $\text{weak}^{*}$ topology.
\end{lemma}

\begin{proof}
The set $\NMalpha$ is tight in $\Mp(0, 1)$ and so its closure is compact by Prokhorov's Theorem. It remains to prove that the set is closed. Suppose $\mu \in \Mp(0, 1)$, $\mu_n \in \NMalpha$ and $\mu_n \xrightarrow{w^*} \mu$. By the Portmanteau Theorem we have that $\mu((0, x)) \leq \liminf \mu_n((0, x)) \leq Mx^{\alpha}$ for every $x$. The analogous inequality holds at the other endpoint.
\end{proof}

The following lemma is a version of \cite[Lemma~2]{CzernousSzarek} uniformized so it holds on a whole neigbourhood of $f$.

\begin{lemma}\label{Lemma:existence_of_epsilon_m_and_aplha}
For arbitrary $f \in \RDSI$ there exist $\epsilon > 0$, $M > 0$, and $\alpha \in (0, 1)$ such that
$$P_g(\NMalpha) \subset \NMalpha\quad \text{for every}\ g \in B_{\epsilon}(f).$$
\end{lemma}

\begin{proof}
The proof is given on a neighborhood of zero. Analogous reasoning applies at the other endpoint with possibly larger $M$ and smaller $\epsilon$ and $\alpha$.

From Corollary~\ref{Cor:xi_is_continuous_function} it follows that $\Lambda_0: \RDSI \to \mathbb{R}$ is a continuous function. Therefore for given $f \in \RDSI$ we can choose $\lambda_0 < f'_0(0)$ and $1 < \lambda_1 < f'_1(0)$ such that
$$\Lambda = p_f \log(\lambda_0) + (1-p_f) \log(\lambda_1) > 0.$$
Transfer $f$ to $F=\phi(f) \in \RDSR$ and choose
$$\epsilon = \min_{i \in \{0, 1\}} \left\{ \frac{\lim_{x \to -\infty}{(F_i(x)-x)} - \log(\lambda_i)}{2} \right\}.$$ 
Continuing in $\RDSR$, there exists $R$ large enough such that 
$$\left|(F_i(x)-x) - \lim_{x\to-\infty}(F_i(x) - x) \right| < \epsilon,\quad \text{for}\ x < -R,\ i \in \{0, 1\}.$$
Consequently for $G \in B_{\epsilon}(F)$, we have
$$G_i(x) > \log(\lambda_i)+x,\quad \text{for}\ x < -R,\ i \in \{0, 1\},$$
which translates back to $\RDSI$ for $g \in B_{\epsilon}(f)$ and letting $x_0 = h^{-1}(-R)$ to
$$g_i(x) > \lambda_i x,\quad \text{for}\ x \in (0, x_0),\ i \in \{0, 1\}.$$
A similar analysis for the inverse $f_i^{-1}$, with the possibility of increasing $R$ and adjusting $x_0 = h^{-1}(-R)$, yields
$$g^{-1}_i(x) < \frac{x}{\lambda_i},\quad \text{for}\ x \in (0, x_0),\ i \in \{0, 1\},$$
considering that the related image $F_i^{-1}$ is just a ``flip over the diagonal."
Let $$\xi(\alpha, \beta) = (p_f+\beta) e^{-\alpha \log(\lambda_0)} + (1-(p_f+\beta)) e^{-\alpha \log(\lambda_1)}.$$ We expand the terms $e^{-\alpha \log(\lambda_i)}$ by their Taylor series at $0$ and get
\begin{align*}
\xi(\alpha, \beta) &= (p_f+\beta)(1 - \alpha\log(\lambda_0) + \mathcal{O}(\alpha^2)) + (1-(p_f+\beta)) (1 - \alpha\log(\lambda_1) + \mathcal{O}(\alpha^2)) \\
&= 1 - \alpha((p_f+\beta) \log(\lambda_0) + (1-(p_f+\beta))\log(\lambda_1)) + \mathcal{O}(\alpha^2).
\end{align*}
For $\beta = 0$, we have 
\begin{align*}
\xi(\alpha, 0) &= 1 - \alpha(p_f \log(\lambda_0) + (1-p_f)\log(\lambda_1)) + \mathcal{O}(\alpha^2) \\
&= 1 - \alpha\Lambda + \mathcal{O}(\alpha^2)
\end{align*}
Fix sufficiently small $\alpha>0$ so that $\xi(\alpha,0)<1$. By continuity of $\xi$, there exists $\delta > 0$ such that $\xi(\alpha, \beta) < 1$ for $\beta \in (-\delta, \delta)$. If $\delta < \epsilon$ then we adjust $\epsilon = \delta$. Choose $M > 0$ such that $$Mx_0^{\alpha} > 1.$$
Let $g \in B_{\epsilon}(f)$, $\mu \in \NMalpha$ be arbitrary. If we assume $x \in (0, x_0)$ then
\begin{align*}
P_g \mu((0, x)) &= p_g\mu((0, g_0^{-1}(x))) + (1-p_g)\mu((0, g_1^{-1}(x))) \\
&\leq p_g\mu\left(\left(0, \frac{x}{\lambda_0}\right)\right) + (1-p_g)\mu\left(\left(0, \frac{x}{\lambda_1}\right)\right) \\
&\leq M \left(p_g \frac{x^{\alpha}}{\lambda_0^{\alpha}} + (1-p_g) \frac{x^{\alpha}}{\lambda_1^{\alpha}} \right) \\
&= M x^{\alpha} (p_g e^{-\alpha \log(\lambda_0)} + (1-p_g) e^{\log(-\alpha \log(\lambda_1))} \\
&= M x^{\alpha} \left( (p_f+\beta) e^{-\alpha \log(\lambda_0)} + (1-(p_f+\beta)) e^{-\alpha\log(\lambda_1)}\right) \\
&= Mx^{\alpha} \xi(\alpha, \beta) \\
&\leq Mx^{\alpha}.
\end{align*}
Conversely, if $x \in [x_0, 1)$ then $P_g\mu((0, x)) \leq 1 \leq Mx_0^{\alpha} \leq Mx^{\alpha}$.
\end{proof}

\begin{corollary}
The stationary measure $\mu_g$ belongs to $\NMalpha$ for every $g \in B_{\epsilon}(f)$ where $M$, $\alpha$ and $\epsilon$ are as in Lemma~\ref{Lemma:existence_of_epsilon_m_and_aplha}.
\end{corollary}

\begin{proof}
By \cite[Theorem~2]{CzudekSzarek}, $P_g^n \mu \xrightarrow{\text{w}^*} \mu_g$ in the weak$^*$ topology for any measure $\mu \in \Mp(0, 1)$, that is the stationary measure is asymptotically stable. Therefore taking $\mu$ from $\NMalpha$ it follows from Lemmas~\ref{Lemma:MNalpha_is_compact} and \ref{Lemma:existence_of_epsilon_m_and_aplha} that $\mu_g \in \NMalpha$.
\end{proof}

\begin{lemma}\cite[Lemma~7]{CzernousSzarek}\label{Lemma:mu_n_admits_sequence_weakly_convergent}
Suppose that the iterated function systems $f \in \RDSI$ and $f_1, f_2, ... \in \RDSI$, are such that, for some $M > 0$ and $\alpha \in (0, 1)$,
$$\mu_1, \mu_2, ... \in \NMalpha\quad \text{and}\quad \lim_{n \to \infty} d_0(f_n, f) = 0,$$
where $\mu_n \in \Mp(0, 1)$ is the unique stationary measure for $f_n$ and
$$d_0(f, g) = |p_f - p_g| + ||f_0 - g_0||_{\sup} + ||f_1 - g_1||_{\sup}.$$
Then $(\mu_n)_n$ admits a subsequence weak$^*$ convergent to a unique stationary distribution $\mu_f$ for $f$; moreover, $\mu_f \in \NMalpha$.
\end{lemma}

\begin{proof}(of Theorem~\ref{Th:mu_start_is_continuous})
Let $f_n$ be a sequence of systems from $\RDSI$ such that $f_n \to f$, that is $d_m(f_n, f) \to 0$. Let $\mu_n$ denote the unique stationary measure for $f_n$ and $\mu_f$ the unique stationary measure for $f$. By Lemma~\ref{Lemma:existence_of_epsilon_m_and_aplha}, there exist $M$, $\alpha$, and $\epsilon$ such that for $g \in B_{\epsilon}(f)$ we have $\mu_g\in \NMalpha$, therefore there exists $N \in \mathbb{N}$ such that $\mu_n \in \NMalpha$ for $n > N$. Let $f_n$ denote the tail of the sequence which is eventually in the $\epsilon$-ball and let $\mu_n$ again denote the associated stationary measures.

Let $g_k = f_{n_k}$ be any subsequence of $f_n$ and $\nu_k = \mu_{n_k}$. Then the sequences $\nu_k$ and $g_k$ satisfy the assumptions of Lemma~\ref{Lemma:mu_n_admits_sequence_weakly_convergent}
\begin{enumerate}[(i)]
  \item $\lim_{k \to \infty}{d_0(g_k, f)} = 0$ for $d_0(g_k, f) \leq 3 d_m(g_k, f)$ by Equation~\eqref{EQ:property_of_d_h} and the definition of $d_m$, and
  \item $\nu_k \in \NMalpha$ for all $k$.
\end{enumerate}
Then $\nu_k$ admits a weak$^*$ convergent subsequence to $\mu_f$, that is $\nu_{k_l} \xrightarrow{w^*} \mu_f$. Since this is true for every subsequence $g_k$ of the sequence $f_n$, we conclude that $\mu_n$ converges to $\mu_f$. 
\end{proof}

\section{Main results}

The continuity of the operator $\psi:\RDSI\to\Mp(0,1)$ which assigns to a system $f$ its unique stationary measure $\mu_f$ allows us to show that the dense sets of systems identified in Section~\ref{sec:dense_subsets} are residual. We use the following general result from Lenz and Stollmann.

\begin{theorem}\label{Fact:Generic_measures}
\cite[Theorem~2.2]{LenzStollmann} 
\begin{enumerate}[(i)]
  \item The set of measures $\mu \in \Mp(0, 1)$ that are singular form a $G_\delta$ subset of $\Mp(0, 1)$.
  \item The set of measures $\mu \in \Mp(0, 1)$ that are fully supported form a $G_\delta$ subset of $\Mp(0, 1)$.
\end{enumerate}
\end{theorem}

\begin{theorem}\label{Th:Main_result}
For the generic system $f \in \RDSI$, the unique stationary measure $\mu_f$ is singular and has full support.
\end{theorem}

\begin{proof}
Let $\mathcal{M}_{\text{sing}} \subset \Mp(0,1)$ denote the set of singular measures $\mu\in\Mp(0,1)$. For $f\in\RDSI$ the stationary measure is $\mu_f=\psi(f)$, and so $\mu_f$ is singular if and only if $\psi(f)\in\mathcal{M}_{\text{sing}}$, that is,
\begin{equation*}
\{ f \in \RDSI : \mu_f\ \text{is singular} \} = \psi^{-1}(\mathcal{M}_{\text{sing}}).
\end{equation*}
By Theorem~\ref{Fact:Generic_measures} (i), the set $\mathcal{M}_{\text{sing}}$ is a $G_\delta$ in $\Mp(0, 1)$. Taking the pre-image of this set under the continuous function $\psi$ we conclude that $\{ f \in \RDSI : \mu_f\ \text{is singular} \}$ is a $G_\delta$ in $\RDSI$. But this set is dense by Corollary~\ref{Cor:Dense_set_with_singular_measure}.

Again let $\mathcal{M}_{\text{f.s.}}$ denote the set of measures $\mu\in\Mp(0,1)$ which have full support. Then
\begin{equation*}
\{ f \in \RDSI : \mu_f\ \text{has full support} \} = \psi^{-1}(\mathcal{M}_{\text{f.s.}}).
\end{equation*}
By Theorem~\ref{Fact:Generic_measures} (ii), the set $\mathcal{M}_{\text{f.s.}}$ is a $G_\delta$ in $\Mp(0, 1)$. Therefore $\{ f \in\RDSI : \mu_f\ \text{has full support}\}$ is a $G_\delta$ in $\RDSI$, and this set is dense by Theorem~\ref{Th:Dense_set_with_full_support}.

This completes the proof, since the intersection of two dense $G_\delta$ sets is again a dense $G_\delta$ set and a dense $G_{\delta}$ subset of a Baire space is residual by definition.
\end{proof}

Together with Theorem~\ref{Th:Existence_of_unique_stationary_measure}, we may now conclude that a typical random dynamical system in $\RDSI$ has a unique stationary measure in $\Mp(0, 1)$ which is singular, non-atomic and has full support.

\subsection*{Data Availability Statement}
Data sharing not applicable to this article as no datasets were generated or analysed during the current study.


\end{document}